\documentclass[10pt]{article}
\usepackage{amsmath}
\usepackage{amsfonts}
\usepackage{amssymb}
\usepackage{mathtools}
\usepackage{geometry}
\usepackage{latexsym}
\usepackage{extarrows}
\usepackage{bm}
\usepackage{braket}
\usepackage{amsthm}
\usepackage{hyperref}
\usepackage{enumitem}
\usepackage[osf]{mathpazo}
\usepackage{graphicx}
\usepackage{enumitem}

\newtheorem{theorem}{Theorem}

\newcommand*{\skal}
[2]{\left\langle\,#1\,\middle|\,#2\,\right\rangle}
\newcommand{\0}{\mathbf{0}}
\newcommand{\R}{\mathbb{R}}

\newgeometry{tmargin=3.5cm, bmargin=3.5cm, lmargin=2.25cm, rmargin=2.25cm}

\begin{document}

\author{Micha\l\ Be\l dzi\'{n}ski, Marek Galewski and Robert Stegli\'{n}ski}
\title{Solvability of abstract semilinear equations by a global
diffeomorphism theorem}
\maketitle

\begin{abstract}
In this work we proivied a new simpler proof of the global diffeomorphism theorem from \cite{Idczak} which we further apply to consider unique solvability of some abstract semilinear equations. Applications to the second order Dirichlet
problem driven by the Laplace operator are given.
\end{abstract}

\begin{center}
Institute of Mathematics Lodz University of Technology\\[0pt]
\L \'{o}d\'{z}, Poland
\end{center}

\section{Introduction}

The idea of applying global invertibility results to boundary value problems
and integral, integro-differential equations has been known in the
literature for some time now. There is a variational tool concerning global invertibility which we are going to use.

\begin{theorem}
	\label{theorem_idczak} 
	\cite[Theorem 3.1]{Idczak}Let $X$ be a real Banach space and let $H$ be a real Hilbert space.
	Suppose that $F:X\rightarrow H$ is a $C^{1}$ mapping such that:
	\begin{enumerate}[label=\textbf{D\arabic*}]
		\item\label{idczak1} for every $y\in H$ functional $\varphi_y:X\to\mathbb{R}$ given by $\varphi_y(x):=\tfrac{1}{2}\|F(x)-y\|^2$ satisfies Palais-Smale condition, i.e. every sequence $(x_n)_{n\in\mathbb{N}}\subset X$
		such that $(\varphi_y(x_n))_{n\in\mathbb{N}}$ is bounded and $\varphi_y^{\prime
		}(x_n)\to\mathbf{0}_{X^*}$ admits convergent subsequence;
		\item\label{idczak2} for every $x\in X$ an operator $F^{\prime }(x)$ is bijective.
	\end{enumerate}
	Then $F$ is diffeomorphism.
\end{theorem}

The proof of Theorem \ref{theorem_idczak} relies on the application of the
celebrated Mountain Pass Theorem due to Ambrosetti and Rabinowitz, see \cite%
{Amb_Rab} and relies in checking that the functional $\varphi $ satisfies
the mountain geometry. Precisely speaking the fact that $f$ is onto is
reached through the classical Ekeland's Variational Principle. The
injectivity part is obtained by contradiction assuming to the contrary and
arguing by the application of the Mountain Pass Theorem. The most difficult
part of the proof is the estimation of $\varphi $ on some sphere around $0$.
However, we will show using some ideas from \cite{pucci} that the proof can
be performed in a different and more readible manner thus simplifying the
arguments from {\cite{Idczak}. }Theorem \ref{theorem_idczak}\ proposes some
approach towards the existence of solutions to nonlinear equations which is
variational in spirit, i.e. concerns the usage of certain functional which
is at the same different from the classical energy (Euler type) action
functional. Moreover, it allows for obtaining uniqueness of a solutions
without any notion of convexity, again contrary to what is known in the
application of a direct method, see for example \cite[Corollary 1.3]{Ma}.
However up to now Theorem \ref{theorem_idczak} and related global implicit
function theorem from \cite{idczakImplicit} have been applied to various
first order integro-differential problems which cover also the so called
fractional case (with the fractional derivative) and correspond to Urysohn
and Volterra type equations, see \cite{bors,idczakWalczak,majewski}. Some
comments on the global invertibility results from \cite{Idczak}, relation
with other approaches and possible applications are contained in \cite{gutu}%
. There was also an attempt to examine second order Dirichlet problem for
O.D.E. in \cite{MBMG}, but for some specific problem and without any
abstract scheme allowing for considering boundary value problems in some
unified manner. Results for continuous problem in \cite{MBMG} are related to
the existence result obtained in \cite{Radu}, although the methods are
different, both yield the existence with similar assumptions. This suggests
that possibly the abstract framework here is to be obtained with some
different global invertibility result. Our applications are meant for
partial differential equations and thus do not have their counterparts in 
\cite{Radu}.

In this work we aim at proposing some abstract approach in order to examine
solvability of some semilinear equations pertaining to second order
Dirichlet problems for both ordinary and partial differential equations
using the approach suggested by Theorem \ref{theorem_idczak}. Our results
towards abstract approach were inspired by some recent abstract approaches
developed in \cite{Budescu,precup2} which were based on the variational
framework due to \cite{precup2} and which utilized relations between
critical points to actions functional and \^{A}\ldots xed points to certain
mappings. Nevertheless, our approach towards solvability is different and
relies on different abstract tools. Morevoer, the setting is now somehow
different since for the sake of global invertiblity densly defined
operators are insuficient. In fact one need to consider the
domain of the operator with its natural topology induced by a suitable norm.

\section{Problem formulation and main results}

Let $(H,\skal{\cdot}{\cdot})$ be a real Hilbert space with a norm denoted by $%
\Vert \cdot \Vert $ and let $A$ be a self-adjoint operator on $H$ with the
domain $D(A)$. Recall that $(D(A),\skal{\cdot}{\cdot}_{A})$ is a real Hilbert
space, where $\skal{\cdot}{\cdot}_{A}=\skal{\cdot}{\cdot}+%
\skal{A\cdot}{A\cdot}$. By $\Vert \cdot \Vert _{A}$ we denote its norm, i.e.
the graph norm of $A$. Let $(B,\Vert \cdot \Vert
_{B})$ be a real Banach space and let $N:(B,\Vert
\cdot \Vert _{B})\rightarrow (H,\skal{\cdot}{\cdot})$ be an operator which
is not necessary linear. In this framework we shall study in $D(A)$\ the
following equation 
\begin{equation}
Ax=N(x)  \label{mainEqu}
\end{equation}

In order to consider (\ref{mainEqu}) we will make the following assumptions:

\begin{enumerate}
\item[\normalfont (A1)] $D(A)\subset B\subset H$ and the embedding $(D(A),%
\skal{\cdot}{\cdot}_{A})\hookrightarrow(B,\|\cdot\|_B)$ is compact;

\item[\normalfont (A2)] $\skal{Au}{u} \geqslant \alpha \| u\|^2$ for some $%
\alpha >0$ and all $u\in D(A)$;

\item[\normalfont (N1)] $N$ is of class $C^1$ with $N^{\prime }(u)$
symmetric for all $u\in B$;

\item[\normalfont (N2)] there exist constants $0<\beta<1$, $0<\gamma<\alpha$%
, $\delta>0$, such that:

\begin{itemize}
\item[\normalfont (i)] $\|N(u)\| <\beta\|Au\|+\delta$ for all $u\in D(A)$;

\item[\normalfont (ii)] $\skal{N'(u)h}{h} <\gamma\|h\|^2 $ for all $u,h\in
D(A)$.
\end{itemize}
\end{enumerate}

Our main result reads as follows.

\begin{theorem}
	\label{theorem_main} Assume that {\normalfont (A1)-(A2)} and {\normalfont%
		(N1)-(N2)} are satisfied. Then equation {\normalfont (\ref{mainEqu})} has a
	unique solution in $D(A)$.
\end{theorem}

In this Theorem we may replace assumption $(A1)$ by the following one:

\begin{enumerate}
	\item[\normalfont(A1')] $(B,\Vert \cdot \Vert _{B})=(H,\skal{\cdot}{\cdot})$
	and $A:D(A)\rightarrow H$ is a self-adjoint operator with purely discrete
	spectrum,
\end{enumerate}

because then the embedding $(D(A),\skal{\cdot}{\cdot}_{A})\hookrightarrow(H,%
\skal{\cdot}{\cdot})$ is compact, by {\normalfont\cite[Proposition 5.12]{S}}.

{\bf Remark.} While all spaces which we consider are real, the theory developled in \cite{S} works for complex spaces. Nevetheless, results which we use (namely: Proposition 3.10,  Proposition 5.12, Proposition 10.19) can be clearly taken to the setting of a real space using the spirit of a book by Brezis \cite{B}. Moreover Kato-Rellich Theorem for the setting of a real space is contained in \cite{Gust}.

\section{Proofs}

For the proof of Theorem \ref{theorem_idczak} we need the following Theorem.
\begin{theorem}
\label{theorem_pucci}
\cite[Theorem 2]{pucci} Let $X$ be a Banach space and let $J:X\to\R$ be a $C^1$ functional satisfying Palais-Smale condition with $\0_X$ its strict local minimum. If there exists $e\neq\0_X$ such that $J(e)\leqslant J(\0_X)$, then there is a critical point $\bar{x}$ of $J$, with $J(\bar{x})>J(\0_X)$, witch is not a local minimum.
\end{theorem}

\begin{proof}[Proof of Theorem \ref{theorem_idczak}]
Firstly, we show that operator $F$ is ,,onto''. Fix $y\in H$. As $F$  is of class $C^1$, $\varphi_y(x)=\tfrac{1}{2}\|F(x)-y\|^2$ is of the same type and its differential $\varphi_y'(x)$ at $x\in X$ is given by
\begin{align*}
\varphi_y'(x)h=\skal{F(x)-y}{F'(x)h}.
\end{align*}
for all $h\in X$.
Clearly, $\varphi_y$ is bounded from below and it satisfies Palais-Smale condition, by \ref{idczak1}. Hence, $\varphi_y$ has a critical point (see \cite[Chapter 3, Corollary 3.3]{Ma}). In other words, there exists $x_0\in X$ such that $\skal{F(x_0)-y}{F'(x_0)h}=0$ for all $h\in X$. Since $F'(x_0)$ is surjective, $F(x_0)-y=0$ and so $F(x_0)=y$.

Now we show that $F$ is ,,one-to-one''. Aiming for a contradiction, suppose that  there exist $x_1,x_2\in X$ such that $x_1\neq x_2$ and $F(x_1)=F(x_2)$. Define $e:=x_1-x_2$ and put $\psi:X\to\R$ by formula 
\begin{align*}
\psi(x):=\tfrac{1}{2}\|F(x+x_2)-F(x_1)\|^2=\varphi_{F(x_1)}(x+x_2).
\end{align*}
Then $\psi$ is of class $C^1$ and $\psi(\0_X)=\psi(e)=0$. Moreover, $\0_X$ is a strict local minimum of $\psi$, since otherwise, in any neighbourhood of $\0_X$ we would have a nonzero $x$ with $F(x+x_2)-F(x_1) = \0_H$ and this would contradict the fact that $F$ defines a local diffeomorphism. Therefore we can apply Theorem \ref{theorem_pucci} and, in consequence, there exists $\bar{x}\in X$ such that $\psi(\bar{x})>0$ and $\psi'(\bar{x})=\0_{X^*}$. Hence 
\begin{align*}
\psi'(\bar{x})h=\skal{F(\bar{x}+x_2)-F(x_1)}{F'(\bar{x}+x_2)h}=0
\end{align*}
for all $h\in X$. Again, by surjectivity of $F'(\bar{x}+x_2)$, we have $F(\bar{x}+x_2)-F(x_1)=0$ and so $\psi(\bar{x})=0$, which contradicts $\psi(\bar{x})>0$. Obtained contradiction ends the proof.
\end{proof}

Now, we can present the proof of the main Theorem.
\begin{proof}[Proof of Theorem \ref{theorem_main}]
By (A2) we have $\Vert Au\Vert \geqslant \alpha \Vert u\Vert $ for $u\in D(A)$ and so
\begin{equation} \label{equiv}
\Vert Au\Vert \geqslant \tfrac{\alpha }{1+\alpha }(\Vert Au\Vert +\Vert u\Vert )\geqslant \tfrac{\alpha }{1+\alpha }\Vert u\Vert _A.
\end{equation}
Let $X:=(D(A),\Vert\cdot\Vert_A)$ and let  the operator $\widetilde{N}:X\rightarrow H$ be defined by $\widetilde{N}=N\circ i$, where $i:X\hookrightarrow (B,\Vert \cdot \Vert_{B})$ is a compact embedding given by (A1). Then $\widetilde{N}\in C^{1}(X,H)$ and operator $\widetilde{N}^{\prime }(u)$ is symmetric compact and linear for all $u\in X$, by (N1). Since $i(u)=u$, any solution of equation
\begin{equation*}
Au=\widetilde{N}(u)
\end{equation*}
is also a solution of equation (\ref{mainEqu}).

Let us define $F:X\rightarrow H$ by 
\begin{equation}
F(u):=Au-\widetilde{N}(u).  \label{formula_F}
\end{equation}
Fix $y\in H$ and consider the mapping $\varphi _{y}:X\rightarrow \mathbb{R}$ given by 
\begin{equation}
\varphi _{y}(u):=\tfrac{1}{2}\Vert F(u)-y\Vert ^{2}.
\label{formula_varphi_y}
\end{equation}
Then $\varphi _{y}\in C^{1}(X,\mathbb{R})$, $F\in C^{1}(X,H)$ and its derivatives are given, respectively, by the following formulas 
\begin{equation*}
\varphi _{y}^{\prime }(u)h=\skal{Au-\widetilde{N}(u)-y}{Ah-\widetilde{N}'(u)h}
\end{equation*}
and 
\begin{equation*}
F^{\prime }(u)h=Ah-\widetilde{N}^{\prime }(u)h
\end{equation*}
for every $u,h\in X$. 

In order to be able to use Theorem \ref{theorem_idczak} we must show that $\varphi _{y}$ satisfies Palais - Smale condition and $F^{\prime }(u)$ is bijective for all $u\in X$.\bigskip

By applying (N2) we see that 
\begin{equation*}
\Vert F(u)-y\Vert =\Vert Au-N(u)-y\Vert \geqslant \Vert Au\Vert -\beta \Vert Au\Vert -\delta -\Vert y\Vert \geqslant \left( 1-\beta \right) \Vert u\Vert_{X}-\delta -\Vert y\Vert
\end{equation*}
for every $u\in X$. This implies that $\varphi _{y}$ is coercive. Thus any (PS) sequence can be assumed to be weakly convergent.\bigskip

Now we show that the functional $\varphi_y $ satisfies (PS) condition on $X$. Assume that $(u_{n})_{n\in \mathbb{N}}\subset X$ is such that:

\begin{itemize}[align=left]
\item[(PS1)] $(\varphi_y (u_{n}))_{n\in \mathbb{N}}$ is bounded;
\item[(PS2)] $\varphi_y^{\prime }(u_{n})\to\mathbf{0}_{X^*}$ if $n\to\infty$.
\end{itemize}

Since $\varphi _{y}$ is coercive, (PS1) shows that $(u_{n})_{n\in \mathbb{N}%
} $ is bounded in $X$, and then after a subsequence, it is weakly convergent
to some $u_{0}\in X$. From (A1) there exists another subsequence, denote it
once again by $(u_{n})_{n\in \mathbb{N}}$, wich is convergent in $(B,\Vert
\cdot \Vert _{B})$. So, by our assumptions we have

\begin{itemize}[align=left]
\item $u_{n}\to u_0$ in $B$;
\item $\widetilde{N}(u_n)\to \widetilde{N}(u_0)$ in $H$;
\item $\widetilde{N}^{\prime }(u_n)\to \widetilde{N}^{\prime }(u_0)$ in ${\mathcal{L}(B,H)}$;
\item $(A(u_n))_{n\in\mathbb{N}}$ is bounded in $H$.
\end{itemize}

Now, a direct calculation yields 
\begin{equation}
\varphi _{y}^{\prime }(u_{n})(u_{n}-u_{0})-\varphi _{y}^{\prime}(u_{0})(u_{n}-u_{0})=\Vert Au_{n}-Au\Vert ^{2}+\sum_{k=1}^{4}\psi_{k}(u_{n}),  \label{equality_PS}
\end{equation}
where 
\begin{align*}
\psi _{1}(u_{n})& =\skal{Au_0-\widetilde{N}(u_0)}{\widetilde{N}'(u_0)(u_n-u_0)}, \\
\psi _{2}(u_{n})& =\skal{\widetilde{N}(u_0)-\widetilde{N}(u_n)}{Au_n-Au_0}, \\
\psi _{3}(u_{n})& =\skal{\widetilde{N}(u_n)-Au_n}{\widetilde{N}'(u_n)(u_n-u_0)}, \\
\psi _{4}(u_{n})& =\skal{y}{(\widetilde{N}'(u_0)-\widetilde{N}'(u_n))(u_n-u_0)}.
\end{align*}
Then, using observations made above, we obtain 
\begin{align*}
|\psi _{1}(u_{n})|& \leqslant \Vert Au_{0}-\widetilde{N}(u_{0})\Vert \Vert \widetilde{N}^{\prime}(u_{0})(u_{n}-u_{0})\Vert \rightarrow 0, \\
|\psi _{2}(u_{n})|& \leqslant \Vert Au_{n}-Au_{0}\Vert \Vert \widetilde{N}(u_{n})-\widetilde{N}(u_{0})\Vert \rightarrow 0, \\
|\psi _{3}(u_{n})|& \leqslant \Vert \widetilde{N}(u_{n})-Au_{n}\Vert \Vert \widetilde{N}^{\prime}(u_{n})(u_{n}-u_{0})\Vert \rightarrow 0, \\
|\psi _{4}(u_{n})|& \leqslant \Vert y\Vert \Vert (\widetilde{N}^{\prime}(u_{0})-\widetilde{N}^{\prime }(u_{n}))(u_{n}-u_{0})\Vert \rightarrow 0
\end{align*}
as $n\rightarrow \infty $. On the other hand, by (PS2) and be the weak convergence of $(u_{n})_{n\in \mathbb{N}}$ to $u_{0}$ in $X$, we have 
\begin{equation*}
|\varphi _{y}^{\prime }(u_{n})(u_{n}-u_{0})|\leqslant \Vert \varphi_{y}^{\prime }(u_{n})\Vert _{X^{\ast }}\Vert u_{n}-u_{0}\Vert_A \rightarrow 0
\end{equation*}
and 
\begin{equation*}
|\varphi _{y}^{\prime }(u_{0})(u_{n}-u_{0})|\rightarrow 0
\end{equation*}
as $n\rightarrow \infty $. Coining the above observations together, we can now show that equality (\ref{equality_PS}) implies 
\begin{equation*}
\Vert Au_{n}-Au_{0}\Vert \rightarrow 0
\end{equation*}
as $n\rightarrow \infty $ which means, by (\ref{equiv}), that $(u_{n})_{n\in \mathbb{N}}$ converges strongly to $u_{0}$ in $X$. This shows that $\varphi _{y}$ satisfies (PS) condition.\bigskip

Now, we show that $F^{\prime }(u)$ is bijective for any $u\in X.$ Fix $u\in X.$ Since $A$ is self-adjoint operator and since $\widetilde{N}^{\prime }(u)$ is a symmetric compact linear operator, it follows that $F^{\prime }(u)$ is self-adjoint operator, by \cite[RKNG Theorem in real Hilbert space]{Gust}. Using (A2) and (N2) we get 
\begin{equation*}
\Vert Ah-\widetilde{N}^{\prime }(u)h\Vert \Vert h\Vert \geqslant \skal{Ah-\widetilde{N}'(u)h}{h}=\skal{Ah}{h}-\skal{N'(u)h}{h}\geqslant \alpha \Vert h\Vert ^{2}-\gamma \Vert h\Vert ^{2}.
\end{equation*}
Hence, equivalently 
\begin{equation}
\Vert F^{\prime }(u)h\Vert =\Vert Ah-\widetilde{N}^{\prime }(u)h\Vert \geqslant (\alpha-\gamma )\Vert h\Vert
\end{equation}
for all $h\in H$. Then, as $F^{\prime }(u)$ is linear, it is injective. Applying \cite[Proposition 3.10]{S}, $F^{\prime }(u)$ is also surjective, and so bijective.

Now we can apply Theorem \ref{theorem_idczak} and obtain a unique $u^{\ast
}\in X$ such that $0=F(u^{\ast })=Au^{\ast }-\widetilde{N}(u^{\ast })$.
\end{proof}

\section{Applications}

As an application of Theorem \ref{theorem_main} we study the following
nonlinear Dirichlet problem 
\begin{equation}
\left\{ 
\begin{array}{l}
-\Delta u(x)=f(x,u(x)), \\ 
\left. u\right\vert _{\partial \Omega }=0.%
\end{array}
\right.  \label{problem_partial}
\end{equation}
Here $\Omega \subset \mathbb{R}^{m}$ is an open and bounded set of class $%
C^{2}$ and $f:\Omega \times \mathbb{R}\rightarrow \mathbb{R}$ is a $C^{1}$%
-Caratheodory function, i.e. for a.e. $x\in \Omega $, $f(x,\cdot )$ is of
class $C^{1}$ and for all $u\in \mathbb{R}$, $f(\cdot ,u)$, $%
f_{u}^{\prime}(\cdot ,u)$ are measurable.\bigskip

An unbounded linear operator $A$ on $H=L^{2}(\Omega )$ defined by $%
Au=-\Delta u$ is self-adjoint if $D(A)=H_{0}^{1}(\Omega )\cap H^{2}(\Omega )$%
, see \cite[Proposition 10.19]{S}. By the Poincar{\' e} inequality 
\begin{equation*}
c_{\Omega }\int_{\Omega }|u(x)|^{2}dx\leqslant \sum_{k=1}^{m}\int_{\Omega
}\left\vert \partial _{k}u(x)\right\vert ^{2}dx
\end{equation*}%
and Green's formula we have 
\begin{equation*}
\skal{A u}{u}\geqslant c_{\Omega }^{2}\Vert u\Vert ^{2},\bigskip u\in D(A),
\end{equation*}%
where $c_{\Omega }$ is a constant in Poincar{\' e} inequality and $%
\skal{\cdot}{\cdot}$ and $\Vert \cdot \Vert $ denote the scalar product and
the norm in $H$, respectively. On $D(A)$ the graph norm of $A$ and norm $%
\Vert \cdot \Vert _{H^{2}(\Omega )}$ are equivalent, see \cite[p. 240]{S}. Therefore if we put 
\begin{equation*}
B_{m}(\Omega ):=\left\{ 
\begin{array}{ll}
C(\overline{\Omega }) & \text{if }m\leqslant 3, \\ 
L^{p_{m}}(\Omega ) & \text{if }m\geqslant 4,%
\end{array}%
\right. 
\end{equation*}%
where $p_{m}>2$ for $m=4$ and $p_{m}\in \left( 2,\frac{2m}{m-4}\right) $ for 
$m>4$, we obtain the compact embedding $(D(A),\skal{\cdot}{\cdot}%
_{A})\hookrightarrow (B_{m}(\Omega ),\Vert \cdot \Vert _{B_{m}})$, see \cite[%
Theorem 1.51]{MMP}. For $m\leqslant 3$ let $c_{m}>0$ be such that $\Vert
u\Vert _{\infty }\leqslant c_{m}\Vert Au\Vert $ for all $u\in D(A)$, where $%
\Vert \cdot \Vert _{\infty }$ denotes the supremum norm.\bigskip 

We will need the following assumptions on $f$:

\begin{itemize}
\item[(P1m)]$ $ 
\begin{itemize}
\item[(if $m\leqslant 3$)] there exist $a_1,b_1\in L^2(\Omega)$, $%
\|b_1\|<c_m^{-1}$ such that $|f(x,u)|\leqslant a_1(x)+b_1(x)|u|$ for a.e. $%
x\in\Omega$ and every $u\in\mathbb{R}$;
\item[(if $m\geqslant 4$)] there exist $a_1\in L^2(\Omega)$ and $%
b_1\in(0,c_\Omega)$ such that $|f(x,u)|\leqslant a_1(x)+b_1|u|$ for a.e. $%
x\in\Omega$ and every $u\in\mathbb{R}$;
\end{itemize}
\item[(P2m)]$ $ 
\begin{itemize}
\item[(if $m\leqslant 3$)] there exist $a_2\in L^2(\Omega)$ and $g\in C(%
\mathbb{R},\mathbb{R})$ such that $|f_u^{\prime }(x,u)|\leqslant a_2(x)g(u)$
for a.e. $x\in\Omega$ and every $u\in\mathbb{R}$;
\item[(if $m\geqslant4$)] there exist $a_2\in L^{q}(\Omega)$ and $b_2>0$
such that $|f_u^{\prime }(x,u)|\leqslant a_2(x) +b_2|u|^r$ for a.e. $%
x\in\Omega$ and every $u\in\mathbb{R}$, where $r=\frac{p_m-2}{2}$ and $q=%
\frac{2p_m}{p_m-2}$;
\end{itemize}

%there exists $a_2\in L^2(\Omega)$ and constant $\alpha>0$ such that $|f'_x(\omega,x)|\leqslant a_2(\omega)+|x|^\alpha$ for a.e. $\omega\in\Omega$ and every $x\in\R$;

\item[(P3)] there exists $b_{3}\in (0,c_{\Omega }^{2})$ such that $%
f_{u}^{\prime }(x,u)<b_{3}$ for a.e. $x\in \Omega $ and every $u\in \mathbb{R%
}$.\bigskip
\end{itemize}

Under the above assumptions the operator $N_{f}:B_{m}(\Omega )\rightarrow
L^{2}(\Omega )$ given by formula $N_{f}(u)(x)=f(x,u(x))$ for $x\in \Omega $
is of class $C^{1}$ with $N_{f}^{\prime }(u)(h)=N_{f_{u}^{\prime }}(u)(h)$
for all $u,h\in B_{m}(\Omega )$. For case $m\geqslant 4$ see \cite[%
Proposition 2.78]{MMP} and for $m\leqslant 3$ see Appendix. Clearly, $%
N_{f}^{\prime }(u)$ is symmetric operator for all $u\in B_{m}(\Omega )$%
.\bigskip

In oder to check (N2), take some $u\in B_{m}(\Omega )$. Using (P1m) we have
for $m\leqslant 3$ 
\begin{align*}
\Vert N_{f}(u)\Vert & =\left( \int_{\Omega }|f(x,u(x))|^{2}\ dx\right) ^{%
\frac{1}{2}}\leqslant \left( \int_{\Omega }|a_{1}(x)|^{2}\ dx\right) ^{\frac{%
1}{2}}+\left( \int_{\Omega }b_{1}(x)|u(x)|)^{2}\ dx\right) ^{\frac{1}{2}} \\
& =\Vert a_{1}\Vert +\Vert b_{1}\Vert \Vert u\Vert _{\infty }\leqslant \Vert
a_{1}\Vert +c_{m}\Vert b_{1}\Vert \Vert Au\Vert .
\end{align*}
and for $m\geqslant 4$ 
\begin{align*}
\Vert N_{f}(u)\Vert & =\left( \int_{\Omega }|f(x,u(x))|^{2}\ dx\right) ^{%
\frac{1}{2}}\leqslant \left( \int_{\Omega }|a_{1}(x)|^{2}\ dx\right) ^{\frac{%
1}{2}}+b_{1}\left( \int_{\Omega }|u(x)|^{2}\ dx\right) ^{\frac{1}{2}} \\
& =\Vert a_{1}\Vert +b_{1}\Vert u\Vert \leqslant \frac{b_{1}}{c_{\Omega }}%
\Vert Au\Vert +\Vert a_{1}\Vert .
\end{align*}
Assumption (P3) provides that for every $u,h\in B_{m}(\Omega )$ there is 
\begin{equation*}
\skal{N_f'(u)h}{h}=\int_{\Omega }f_{u}^{\prime }(x,u(x))h(x)h(x)\
dx\leqslant b_{3}\int_{\Omega }|h(x)|^{2}\ dx=b_{3}\Vert h\Vert ^{2}.
\end{equation*}
As a conclusion, we obtained

%\cite[Proposition 10.19]{S}. \cite[Proposition 10.20]{S} 

\begin{theorem}
Assume that $f:\Omega\times\mathbb{R}\to\mathbb{R}$ is a $C^1$-Caratheodory
function such that {\normalfont(P1m)}, {\normalfont(P2m)} and {\normalfont%
(P3)} hold. Then problem {\normalfont(\ref{problem_partial})} has an unique
solution in $H_0^1(\Omega)\cap H^2(\Omega)$.
\end{theorem}

As an example the following problem 
\begin{equation}
\left\{ 
\begin{array}{l}
-\Delta u(x)=\left( 1-\frac{1}{|x|^{2}}\right) (cu(x)-1), \\ 
\left. u\right\vert _{\partial \Omega }=0.%
\end{array}%
\right. 
\end{equation}%
where $\Omega \subset \mathbb{R}^{3}$ is any open and bounded set of class $%
C^{2}$, $c>0$ is a suitable constant and $|\cdot |$ denotes the Euclidean
norm, has an unique solution in $H_{0}^{1}(\Omega )\cap H^{2}(\Omega )$.

\section{Appendix}

In this appendix we show that for $m\leqslant 3$ if $f:\Omega \times \mathbb{%
R}\rightarrow \mathbb{R}$ is a $C^{1}$-Caratheodory function such that {%
\normalfont(P1m)}, {\normalfont(P2m)} hold, then an operator $N_{f}:C(%
\overline{\Omega })\rightarrow L^{2}(\Omega )$ given by formula $%
N_{f}(u)(x)=f(x,u(x))$ for $x\in \Omega $ is of class $C^{1}$ with $%
N_{f}^{\prime }(u)(h)=N_{f_{u}^{\prime }}(u)(h)$ for all $u,h\in C(\overline{%
\Omega })$. By Theorem B in \cite{Besov}, if $|f(x,u)|\leqslant a(x)g(u)$
for all $x\in \Omega $ and $u\in \mathbb{R}$ with $a\in L^{2}(\Omega )$ and $%
g\in C(\mathbb{R},\mathbb{R})$, then $N_{f}$ is continuous from $C(\overline{%
\Omega })$ into $L^{2}(\Omega )$.\bigskip

First, we show that for all $u,h\in C(\overline{\Omega })$, $%
N_{f_{u}^{\prime }}(u)(h)\in L^{2}(\Omega )$. Indeed, we have 
\begin{align*}
\Vert N_{f_{u}^{\prime }}(u)(h)\Vert & =\left( \int_{\Omega
}|f_{u}^{\prime}(x,u(x))h(x)|^{2}dx\right) ^{\frac{1}{2}}\leqslant \Vert
h\Vert _{\infty}\left( \int_{\Omega }|f_{u}^{\prime }(x,u(x))|^{2}dx\right)
^{\frac{1}{2}} \\
& \leqslant \Vert h\Vert _{\infty }\left(
\int_{\Omega}|a_{2}(x)g(u(x))|^{2}dx\right) ^{\frac{1}{2}}\leqslant \Vert
h\Vert _{\infty }\Vert a\Vert \sup_{x\in \Omega }|g(u(x))|<\infty.
\end{align*}
Now, fix $u\in C(\overline{\Omega })$ and let 
\begin{equation*}
w(h)=N_{f}(u+h)-N_{f}(u)-N_{f_{u}^{\prime }}(u)(h)
\end{equation*}
for all $h\in C(\overline{\Omega })$. Let $h\in C(\overline{\Omega })$. We
have 
\begin{equation*}
\displaystyle f(x,u(x)+h(x))-f(x,u(x))=\int_{0}^{1}f_{u}^{\prime}(x,u(x)+%
\tau h(x))h(x)\ d\tau
\end{equation*}
Hence, using Fubini's theorem, we obtain 
\begin{align*}
\Vert w(h)\Vert & =\left( \int_{\Omega }\left\vert
f(x,u(x)+h(x))-f(x,u(x))-f_{u}^{\prime }(x,u(x))h(x)\right\vert
^{2}dx\right) ^{\frac{1}{2}} \\
& \leqslant \Vert h\Vert _{\infty }\left( \int_{\Omega
}\left\vert\int_{0}^{1}f_{u}^{\prime }(x,u(x)+\tau h(x))\ -f_{u}^{\prime
}(x,u(x))\ d\tau \right\vert ^{2}dx\right) ^{\frac{1}{2}} \\
& \leqslant \Vert h\Vert _{\infty }\left(
\int_{0}^{1}\int_{\Omega}\left\vert f_{u}^{\prime }(x,u(x)+\tau h(x))\
-f_{u}^{\prime }(x,u(x))\right\vert ^{2}dx\ d\tau \right) ^{\frac{1}{2}}
\end{align*}
Since $N_{f_{u}^{\prime }}:C(\overline{\Omega })\rightarrow L^{2}(\Omega)$
is continuous, the above implies that 
\begin{equation*}
\frac{\Vert w(h)\Vert }{\Vert h\Vert _{\infty }}\rightarrow 0
\end{equation*}
as $\Vert h\Vert _{\infty }\rightarrow 0$. The continuity of the map $%
N_{f}^{\prime }:C(\overline{\Omega })\rightarrow \mathcal{L}(C(\overline{%
\Omega }),L^{2}(\Omega ))$ follows now from the continuity of $%
N_{f_{u}^{\prime }}:C(\overline{\Omega })\rightarrow L^{2}(\Omega )$.
Indeed, suppose that $u_{n}\rightarrow u_{0}$ in $C(\overline{\Omega })$.
Then 
\begin{align*}
\Vert N_{f}^{\prime }(u_{n})-N_{f}^{\prime }(u_{0})\Vert _{\mathcal{L}(C(%
\overline{\Omega }),L^{2}(\Omega ))}& =\sup_{\Vert h\Vert _{\infty}=1}\left(
\int_{\Omega }|f_{u}^{\prime
}(x,u_{n}(x))h(x)-f_{u}^{\prime}(x,u_{0}(x))h(x)|^{2}\ dx\right)^{\frac{1}{2}%
} \\
& \leqslant \left( \int_{\Omega }|f_{u}^{\prime
}(x,u_{n}(x))-f_{u}^{\prime}(x,u_{0}(x))|^{2}\ dx\right) ^{\frac{1}{2}%
}\rightarrow 0
\end{align*}
as $n\rightarrow 0$.

\begin{tabular}{l}
Micha\l\ Be\l dzi\'{n}ski \\ 
Institute of Mathematics, \\ 
Lodz University of Technology, \\ 
Wolczanska 215, 90-924 Lodz, Poland, \\ 
\texttt{beldzinski.michal@outlook.com}%
\end{tabular}

\bigskip

\begin{tabular}{l}
Marek Galewski \\ 
Institute of Mathematics, \\ 
Lodz University of Technology, \\ 
Wolczanska 215, 90-924 Lodz, Poland, \\ 
\texttt{marek.galewski@p.lodz.pl}%
\end{tabular}

\bigskip

\begin{tabular}{l}
Robert Stegli{\' n}ski \\ 
Institute of Mathematics, \\ 
Lodz University of Technology, \\ 
Wolczanska 215, 90-924 Lodz, Poland, \\ 
\texttt{robert.steglinski@p.lodz.pl}%
\end{tabular}

\end{document}